\documentclass[12pt]{amsart}

\usepackage{amstext,amssymb,amsmath,stmaryrd,amsthm}
\usepackage{geometry} % see geometry.pdf on how to lay out the page.  
\usepackage{color}
\usepackage{hyperref}
\makeindex

\numberwithin{equation}{section}

\newtheorem{theorem}{Theorem}[section]

\newtheorem{proposition}[theorem]{Proposition}

\newtheorem{example}{Examples}[section]

  \newtheorem{corollary}[theorem]{Corollary}
\newtheorem{remark}[theorem]{Remark}

\usepackage{amssymb}
\usepackage{amsmath}
\usepackage{color}
\usepackage{graphicx}
\usepackage{comment}
\usepackage{float}
\usepackage{hyperref}

\def\N{{\mathbb N}}

\def\I{\textbf{I}}

\newcommand{\R}{\mathbb{R}}

\newcommand{\E}{\mathbb{E}}

\newcommand{\Ge}{\mathcal{L}}

\begin{document}

\title[ Replica Mean Field limits for neural networks ]{ Replica Mean Field limits for  neural networks with excitatory and inhibitory activity}

\author{ Ioannis Papageorgiou }

 \thanks{\textit{Address:}  Universidade Federal do ABC (UFABC) - CMCC, Avenida dos Estados, 5001 - Santo Andre - Sao Paulo, Brasil.
\\ \text{\  \   \      } 
\textit{Email:}  i.papageorgiou@ufabc.edu.br, papyannis@yahoo.com  }


%\thanks{}


\begin{abstract} We study Replica Mean Field limits for a neural system of infinitely many neurons with both inhibitory and excitatory interactions. As a result we obtain an analytical characterisation of the invariant state. In particular we focus on the Galves-L\"ocherbach  model with interactions beyond the Uniform Summability Principle.
\end{abstract}
 
\date{}
\maketitle 
 
\section{Introduction}

 We study infinitely large neural networks. The activity of each neuron is described by the evolution of its membrane potential (see \cite{Tuc}), which is characterized by  a brief depolarization called a spike or a jump (see \cite{H-K-L}, \cite{Kr}). When a neuron spikes, it emits an action potential (see \cite{L-M}). If a neuron $i$ has membrane potential value $x_i$,  it then spikes at a rate $\phi(x_i)$, where the functions $\phi:\R_+ \rightarrow \R_+$ are called intensity functions. When a neuron spikes its potential is reset to a no-negative resting value $r_i$, while neurons that interact with $i$     receive an additional amount of either negative or positive potential which is added to their membrane potential.

In this paper, we study the Replica Mean Field limits of a system of infinitely many inhibitory and excitatory neurons.

One can study  neural networks by focusing on the behavior of individual neurons, as in  \cite{D-L-O}, \cite{C17},  \cite{H-L}, \cite{G-L}   and \cite{H-R-R}, where the attention lies on the examination of the jump times. In this case Hawkes processes are usually used which are memoryless since the point process reset to a fixed value  after a spike.   

 Alternatively,  as in the current paper,  the interactions among neurons in the network can be studied. In this case the entire evolution of the membrane potential of each neuron is studied.  In fact, the evolution of the network between jumps is continuously analysed.  Inhibition models for a finite number of neurons based on \cite{Kr} are studied in \cite{K-M-R},    \cite{Tu} and \cite{Cot92}, where a drift is not necessary since the neuron`s membrane potential   is reduced whenever a spike occurs in the system. The opposite of course happens in the case of excitatory neurons where a drift is crucial for the control of the system. 
 
 The case of networks with   inhibitory and excitatory connections is studied by  \cite{TUROVA1997197},  \cite{Pa} and   \cite{P-R-R}, while  only inhibitory or excitatory connections are studied by \cite{Me-Mo-Tu94}. In \cite{H-K-L},  a network with a deterministic drift has been considered.

The system of neurons that we consider in the current work contains  infinite many neurons that interact with each other through chemical synapses.  Interactions between neurons occur when a neuron spikes. When a neuron $i$ spikes, any neuron $j$ that is connected to $i$ receives an additional amount  $w_{ij}\in \R$ of membrane potential called a synaptic weight.   As for the membrane potential of the neuron $i$ that spikes, it is reset to some no-negative value $r_i$. This description refers to the  Galves-L\"ocherbach model (G-L) introduced in \cite{G-L} to describe   interactions of   neural networks. In \cite{G-L}, the authors required the synaptic weights to satisfy the Property of Uniform Summability (PUS):

\begin{align}\label{Dob}\sup_{i\in \N} \sum_{j\in \N} \vert w_{ji} \vert<\infty.\end{align}

This is a  Markovian Process that belongs in the family of the Piecewise Deterministic Markov Processes introduced in \cite{Davis84} and \cite{Davis93}, which is frequently used in  modeling   chemical and biological phenomena (see for instance  \cite{ABGKZ}, \cite{PTW-10} and \cite{C-D-M-R}). Piecewise Deterministic Markov Processes (PDMPs) have emerged as a powerful mathematical framework for modeling neuronal dynamics, particularly in  integrate-and-fire models like Stein’s neuronal model (\cite{Stein}). The hybrid nature of PDMPs, combining deterministic evolution with stochastic jumps, makes them particularly well-suited for capturing the behavior of neuronal membrane potentials, where continuous voltage dynamics are interrupted by discrete spike resets. PDMPs extend Stein’s model by incorporating random jump processes that represent synaptic input events, stochastic ion channel behavior, or external perturbations. For example, in \cite{La-Mus91}, PDMPs were used to model neurons with complex dendritic structures. The membrane potential at the trigger zone is modeled by a one-dimensional stochastic process with deterministic evolution between stochastic synaptic input events, which provides a synthesis of Stein's concept of the membrane potential behavior and the diffusion concept. Another application of PDMPs in neuroscience is presented in \cite{Re-Th-Tre17}, where  neuronal activity is  formulated as an     optimal control problem for infinite-dimensional PDMPs, incorporating both deterministic membrane potential evolution and stochastic transitions between different states of the ion channel. 

Furthermore, hydrodynamic and mean-field limits of the model were studied in  \cite{Co}, \cite{D-L}, \cite{C-D-L-O}  and \cite{Co2}, while    properties of the invariant measure were studied in  \cite{L17}. In  \cite{Pa2} and \cite{H-P}, some inequalities were also obtained. Phase transition and metastability were studied in \cite{F-G-L}, \cite{A}, \cite{A-N-R}, \cite{A1}, and \cite{A-P}. In \cite{Kr} and \cite{Lo}, no deterministic interactions have been considered.

 In this paper, we study   the Galves-L\"ocherbach  model in the framework of \cite{Pa} that allows interaction weights that go  beyond the restrictive Property of Uniform Summability  of the G-L model. In particular, we study the Replica Mean Field properties for the model beyond uniform summability. Replica Mean   Field (RMF) limits were used by Baccelli  and Taillefumer in \cite{B-T}  to describe the stationary state of a finite  neural network. Since \cite{B-T} considered  finitely many neurons, the system directly satisfied the uniform summability property. We will use the exact same approach here, but we will take advantage of the existence of inhibitory connections cancelling excitatory ones to obtain results about a system of infinitely many neurons that go beyond the Uniform Summability Principle.

The importance of studying neural networks that include both excitatory and inhibitory interactions follows from the fact that the lack of balance between the two types of interaction is associated with certain brain diseases (see \cite{C-M-T-T15}, \cite{BHJOT}). This takes place mainly when excitation dominates over inhibition (see \cite{A-E-V94}, \cite{Br2000}). For this reason, the presence of inhibition is crucial in maintaining balance  (see \cite{Br-Ha99}, \cite{Ma-La96}, \cite{Je-Tr-Wh95} and \cite{Am-Br97}).   Thus, when inhibitory neurons are included,  the average local excitation can balance out with local inhibition (\cite{Se-Ts95}, \cite{Am-Br97}, \cite{Sh-Ne94}, \cite{Lu-Om-Si06}, \cite{Abe91},\cite{Am-Br97}, \cite{So-Vr96}).  
 What is, however, important  to notice is that this balance is not required only in a local level but globally within the network. This means that synapses between distant neurons should also be allowed (see \cite{So-Vr96}, \cite{Am-Br} and \cite{Am-Br97}).

 \section{Replica Mean Field Limits}\label{RMF}
 Replica Mean   Field (RMF) limits were used by Baccelli  and Taillefumer in \cite{B-T} in order to describe the stationary state of a system of finite number of neurons (see also \cite{BaTa20}, \cite{r8},  \cite{r10},  \cite{r49},   \cite{r60} and \cite{Co-Pa}).  In RMF, if we have a system of $K$ neurons,  we consider a number of replicas of this system. If the RMF consists of  $M$ replicas,  every one of them represents a copy of the initial neural system. As a result, there exists a copy of any of the $i\in K$ neurons in every replica. When a neuron spikes, it will interact with neurons from other replicas. Assume that a neuron $i$ interacts with another neuron $j$ in the initial system. In the RMF, when the neuron $i\in K$ from replica $n$ spikes,  then it will   choose uniformly one of the $M-1$ copies of $j$ from any of the other $M-1$ replicas different from $n$ to interact.  The RMF framework provides a biologically plausible way to analyze large, intensity-based neural networks while preserving key finite-size effects. Unlike classical mean-field approximations, which assume infinitely large homogeneous networks, the replica approach introduces independent copies (or replicas) of a finite network with randomized connectivity. This method is particularly relevant for biological neural networks, where sparse connectivity, stochastic activity, and finite-size effects play a crucial role. By incorporating randomness into neural interactions, the RMF model better reflects real-world neural architectures, such as cortical microcircuits and hippocampal networks. In addition, it provides a robust framework for studying synaptic plasticity, learning, and information processing under realistic neural conditions.

 The linear G-L model considered in \cite{B-T}, given that it consists of finitely many neurons, directly satisfies the PUS condition (\ref{Dob}). We will use the exact same approach here, but with some    modifications that will allow us to obtain results about a system of infinitely many neurons.

 The idea behind the RMF structure is to express the moments of the invariant measure in terms of structural characteristics of the network, such  as the drift $a_i$, the reset value $r_i$ and the synaptic weights $w_{ij}$. In particular our aim is to obtain expressions of the moments of spiking rates $\beta_i^k=\E[N_i((0,1])^k]$ in terms of the aforementioned features. To do so, we will derive an ODE for $\beta_i^k$ through the RMF limit.  
 
 The results presented in the current  article as well as the methods to obtain them, follow very closely the work in \cite{B-T}. 
 For a  complete and comprehensive analysis of the RMF approach, one should look at \cite{B-T}. Here, we will   mention only the absolute necessary  aspects of the method  that  will be used. In general, Replica Mean Field models consist of a finite number of copies of the initial model which replaces the initial interactions with higher level inter-replica interactions. Then, assuming the Poisson Hypothesis, when taking the limit of the number of replicas going to infinity we consider independence between the replicas. In the actual schema considered in \cite{B-T}, every replica consists of a finite number of neurons. In our case however, where a network of infinitely many neurons is studied, we will consider replicas of infinitely many neurons. 
 
 The finite-replica model of infinitely many neurons  is defined analogously to the finite-neurons case, as follows:
 
 For a finite model consisting of $M$ replicas, let  $m\in \{1,...,M\}$ and $i\in\N$, $N_{m,i}$ denote the spiking activity of the neuron $(m,i)$, that is the neuron $i$ belonging to the replica $m$. Accordingly, $\phi(x_i^m)$   for $m\in \{1,...,M\}$ and $i\in\N$, represents the spiking activity of the  neuron $(m,i)$, which contrary to the typical model where one neuron $i$ interacts with any neuron $j$ of its network through a synaptic weight $w_{ij}$, in the $M$ replica model, every neuron $(m,i)$ of a replica $m$ interacts only with the neurons $(n,j)$ belonging only to other replicas, i.e. $n\neq m$, through the same intensity weight $w_{ij}$, in such a way that the replica $n$ is chosen in a uniform way among the $M-1$ replicas  $\{1,...,m-1,m+1,...M\}$ different from $m$.  
 
 Although we consider a network of infinitely many neurons, we will assume that for any  neuron $i\in \N$, there is only a finite number $K_i$ of neurons that receive a synaptic weight different from zero from $i$
 \[K_i:=\{j\in \N: w_{ij}\neq 0\}<\infty.\]
 Equally we assume that for a neuron $i\in \N$, there is only a finite number $K^i$     of neurons from which it receives a synaptic weight different from zero, that is:
  \[K^i:=\{j\in \N: w_{ji}\neq 0\}<\infty.\]
  Furthermore, we will  assume that
for every  $i\in \N,$ \[K^{i}\subset \{1,...,i-1,i+1\}.\] 
In essence this is a feedforward neural network, where the succeeding neuron $i+1$ has also been added, in order to avoid  trivial cases under hypothesis (H0) presented below.  Feedforward neural networks, where each neuron receives input from all preceding neurons, are fundamental to various biological systems and play crucial roles in sensory processing, motor control, and neural development. These structures facilitate sequential information transmission, enabling precise timing and efficient computation. In particular, synfire chains, which are characterized by precisely timed neural sequences, have been observed in motor and sensory circuits, contributing to processes such as speech and movement coordination (\cite{Abe91}).

One striking example of a generation of a sequential neural network   occurs in songbirds, where synfire chains in the High Vocal Center  drive sequential firing during song production (\cite{Fe-Ha-Ko02}). Similarly, repetitions of spontaneous synaptic input patern in neocortical neurons in vivo and in vitro were studied  in \cite{A-A-C-L-F-I-Y04}.

Since the set of neurons $K^i$ that influences a neuron $i$ is always finite, our system is infinite-dimensional in the weak sense that $K^i$ can increase to infinity as $i$ goes to infinity. As we will see later in the section, when we study ergodicity this is important in order to talk about a single neuron's stationarity. Furthermore, it should be noted that we  do not assume that the size of the sets $K_i$ or $K^i$ are uniformly bounded on $i$.

\section{Linear intensity  and  linear interaction.}\label{RMF1}

In this section, we study the RMF limit for the linear G-L model, where the intensity function  is linear  $\phi(x_i)=x_i$. In order to stay closer to the   stochastic framework
in  \cite{B-T},  we introduce a variation in the model of the previous section, that is, that   when a neuron $i$
spikes, the membrane potential   does not settle necessarily to $0$ as in \cite{G-L}, but   to some value $r_i\geq 0$. We will consequently study both the cases where $r_i=0 \ \forall i \in \N$ and $r_i>0 \  \forall i\in \N$.

In most classical neuronal models, it is assumed that, following an action potential, the membrane potential is reset to either a hyperpolarised value or to its resting potential -- represented as zero in our model -- that is, $r_i = 0$. This ensures that the neuron remains inactive for a period, commonly referred to as the refractory period. However, evidence suggests that the membrane potential does not always return exactly to its resting value after a spike. In \cite{CostaRibSan1991}, the authors investigated the afterpotential characteristics and firing patterns of developing rat hippocampal CA1 neurons. These neurons undergo substantial electrophysiological changes during postnatal development, particularly in their resting membrane potential (RMP). Their results showed that younger neurons exhibit a less negative RMP compared to mature neurons. A similar observation was reported in \cite{McCPr87}, which examined postnatal cerebral cortical pyramidal neurons in rats; however, in that study, the authors attributed the observed depolarisation to a measurement artefact. More recently, in \cite{Borris13}, the authors analysed a large database of leech heart interneuron models and found that many neurons exhibit multistability. This means that a single neuron model can display multiple stable behaviours, such as quiescence and bursting, including different quiescent resting states. These findings highlight how neurons may leverage multistability for functional flexibility, but also how such dynamics could contribute to dysfunction.

The generator of the process $X$ is given  for any test function $ f : \R_+^\N \to \R $  and $x \in \R_+^{\N}$ by    
\begin{align}\label{genRMF}
\Ge f (x ) =&-\sum_{i\in \N}a_ig(x_i)\frac{d}{dx_i}f +\sum_{i\in \N }   \phi (x_i) \left[ f ( \Delta_i ( x)  ) - f(x) \right]  
\end{align}
for $g : \R_+ \to \R_+$ and $a_i\geq 0$ for all $i\in \N$
where
\begin{equation}\label{deltaRMF}
(\Delta_i (x))_j =    \left\{
\begin{array}{ll}
\max\{ x_{j} +w_{i  j},0\}  & j \neq i \\
r_{i}  & j = i 
\end{array}
\right\}.\end{equation}
   In order to form the infinitesimal generator of the $M$ replica model for the (\ref{genRMF})-(\ref{deltaRMF}) model we need to introduce the set $V_{m,i}$ of neurons that are targeted by the neuro $(m,i)$. When a neuro $j$  in replica $n$, denoted $(n,j)$, spikes, it interacts by sending a weight $w_{jz}$ to every neuro $z\in K_j$. Since in the replica model there are $M-1$ different neurons $z$ it chooses one of the $M-1$ replicas uniformly, say $s$, and interacts with this $(s,z)$ neuro. Since every neuro $(m,i)$ can be targeted by $K_i$ different neurons, $V_{m,i}$ has  cardinality $\vert V_{m,i}\vert=(M-1)^{K_{i}},$ where
\[V_{m,i}=\left\{   v\in \{1,...,M\}^{K_{i}} : v_i=m   \text{\  and\  } v_j \neq m, \ j\neq i \right\}.\]

The infinitesimal generator for the M-replica Markovian dynamics is
 \begin{align*}
 \Ge^{M} [f_u](x)=&-\sum_{i\in\N}\sum_{m=1}^M a_ig(x^{m}_i)\frac{d}{dx_i^m}f_u(x) +\\ &+\sum_{i\in\N}\sum_{m=1}^M\frac{1}{\vert V_{m,i} \vert}\sum_{v\in V_{m,i}}\phi( x_i^{m})\left[ f (x+w_{m,i,v}(x)  ) - f(x) \right] 
 \end{align*}
where the weights distributed after neuron $(m,i)$ spikes are
\[ [w_{m,i,v}(x)]_{j,v_j}=  \left\{
\begin{array}{lcl}
w_{ij} &   \text{ if }    j \neq i, \ v_j  \neq m  \\  
r_i-x_i^m   & \text{ if }    j = i, \  v_j  =m   \\
0 & \text{   otherwise }  
\end{array}
\right\}.  \]
We consider two main sets of different hypothesis  about the  drift and the intensity functions $\phi,g$:
 \begin{itemize}

\item   \underline{Hypothesis (H1):}  The linear (G-L) model as in \cite{B-T}.
\[ \phi(x_i)=x_i,  \ g(x_i)=x_i   \  \text{and} \ \min_{i\in \N} r_i>0. \]
 
 \item    \underline{Hypothesis (H2):}  For   $z> 0$, assume,   
\[ \phi(x_i)=x_i+z,  \ g(x_i)=x_i .   \]
  \end{itemize}

The computer  simulations presented by C. Pouzat (see \cite{PouzM18})  for the G-L model designed to emulate the stochastic dynamics of neural networks, present a firing behavior of neurons that is driven by a rate function, which is conceptually similar to the functions $\phi(x_i)$ in the hypotheses (H1) and (H2).  In Hypothesis (H1),  the firing rate is directly proportional to the membrane potential, suggesting a linear relationship between the neuron's activity and its current state. This linear assumption is often used in theoretical models of neural networks, including the G-L model, to simplify the analysis of network dynamics.  On the other hand, Hypothesis (H2) introduces a shift in the firing rate.

For any set $S\subset \{1,...,M\}$ and $K\subset \N$, define 
   $V^{K,S }_{u}(x)=e^{u\sum_{i\in K}\sum_{m \in S} x_i^m}$. Then 
\begin{align*}
\Ge^{ M} [V^{K,S}_{u}](x)=&-\sum_{i\in\N}\sum_{m\in S} a_ig(x_i^m)\frac{d}{dx_i^m}V^{K,S}_{u}(x) +\\ &+\sum_{i\in\N}\sum_{m=1}^M\frac{1}{\vert V_{m,i} \vert}\sum_{v\in V_{m,i}} \phi(x_i^m)\left[ V^{K,S}_{u} (x+w_{m,i,v}(x)  ) - V^{K,S}_{u }(x) \right]. \end{align*}Since,  a neuron $(j,m)$ with $j\notin \cup_{i\in K}K^{i}$, does not  contribute synaptic weights to the sum $\sum_{i\in K}\sum_{m \in S} x_i^m$, we can write 
\begin{align*}
\Ge^{ M}& [V^{K,S}_{u}](x)=-\sum_{i\in K}\sum_{m\in S} a_ig(x_i^m)uV^{K,S}_{u}(x) +\\  &+\sum_{i\in K}\sum_{m\in S} \frac{1}{\vert V_{m,i} \vert}  \sum_{v\in V_{m,i}}\phi(x_i^m)\left[ e^{u(r_{i}-x_i^m)+\sum_{ j\in K\cap K_{i},u_j\neq m}uw_{ij}} - 1 \right]V^{K,S}_{u}(x) 
+\\  &+\sum_{i\in K^{c}\cap \{ \cup_{j\in K}K^{j}\}}\sum_{m=1} ^{M}\frac{1}{\vert V_{m,i} \vert}  \sum_{v\in V_{m,i}}\phi (x_i^m)\left[ e^{\sum_{j\in K\cap K_{i},u_j\neq m}uw_{ij}} - 1\right]V^{K,S}_{u}(x) 
+\\  &+\sum_{i\in K}\sum_{m\in S^{c}}\frac{1}{\vert V_{m,i} \vert}\sum_{v\in V_{m,i}} \phi (x_i^m)\left[ e^{\sum_{j\in K\cap K_{i},u_j\neq m}uw_{ij}} - 1\right]V^{K,S}_{u}(x). 
\end{align*}
Next we consider $\phi$ of polinomial order, that is $\phi(x_i)=(x_i)^ r+z$ for some positive $r\in \N$ and $z>0$.

 \begin{proposition}Assume that $\phi$ is an increasing function of polynomial order. For  every  finite set of neurons $K=\{1,2,...,s\}$ for some $s\in \N$, such that for $ K^{i}\subset \{1,...,i-1,i+1\}$  for any  $i\in K$ and $u>0,$ there exist $c>0$ and $d>0$  such that 
 \[ \Ge^{M} [V^{K,M}_{u}](x)\leq -c   V^{K,M}_{u}(x)+d\I_{B}(x).   \]
  \end{proposition}
 \begin{proof}Since $K^{i}\subset \{1,...,i-1,i+1\}$ we have that $\{\cup_{j\in K}K^{j}\}\subset \{1,2,...s+1\}$, and so
 \begin{align*} 
\Ge^{M} [V^{K,M}_{u}](x)=&-\sum_{i\in K}\sum_{m=1}^Ma_i g(x_i^m)uV^{K,M}_{u}(x) +\\  &+\sum_{i\in K}\sum_{m=1}^M\frac{1}{\vert V_{m,i} \vert}\sum_{v\in V_{m,i}} \phi(x_i^m)\left[ e^{ u(r_i-x_i^m)+\sum_{j\in K\cap K_{i},u_j\neq m}uw_{ij}} - 1 \right]V^{K,M}_{u}(x) . 
\end{align*}
Since $\phi$ is of polynomial order, for  every $u>0$ and $x\geq 0$, there exists an $   \tilde x\in [0,\infty)$ such that $\phi(\tilde x)e^{-u\tilde x}\geq \phi(x_i)e^{-ux_i}$ for all $x_i\geq b$ for some $b>0$. We can compute 
  \begin{align*}
\Ge^{M} &[V^{K,M}_{u}](x)\leq  \\   & \sum_{i\in K}\sum_{m=1}^M\frac{1}{\vert V_{m,i} \vert}\sum_{v\in V_{m,i}}\left[  \phi(\tilde x)e^{+ur_i-u \tilde x+u\sum_{j\in K\cap K_{i},u_j\neq m}w_{ij}} -  \phi(x_i^m) \right]V^{K,M}_{u}(x)
 . 
\end{align*}
  If for some $c>0$, we consider the compact set 
 \begin{align*}&R^{KM}_c:= \\  &\left\{ x\in \R_+^{KM} ,x_i>b: \sum_{i\in K}\sum_{m=1}^M \phi( x_i^m)\leq \sum_{i\in K}\sum_{m=1}^M\frac{1}{\vert V_{m,i} \vert}\sum_{v\in V_{m,i}}\phi(\tilde x)e^{+ur_i-u \tilde x+u\sum_{j\in K\cap K_{i},u_j\neq m}w_{ij}}  +c \right\},\end{align*}
then,  outside  $R^{KM}_c$ we have 
$\Ge^{K,M} V^{K,M}_u(x)\leq -c  V^{K,M}_u(x) $, while, for $x\in R_c^{KM}$ we get that
 $
\Ge^{K,M} V^{K,M}_{u,1}(x)\leq  -cV^{K,M}_{u,1}+d
,$
for two positive constants $c$ and $d$.
From the two bounds 
 \begin{align*}
\Ge^{M} V^{K,M}_{u} \leq  -cV^{K,M}_{u}+d\I_{R_c^{KM}}
.
\end{align*}

\end{proof}
\begin{remark}In \cite{B-T}, the Foster-Lyapunov  inequality was used to prove ergodicity for a Replica model which consists of a system of finite neurons, so that the MGF of all the neurons in the system with respect to the stationary  measure can be calculated. 

 However, in our case,  we will limit ourselves to one neuron, with the goal of studying an ODE for  $\Lambda^{m,i}(u)=\E^{\{1,..,i+1\}}[e^{u(x_{ i})}],$ where $\E^{\{1,..,i+1\}}$   the stationary measure referring to the partial network containing only the first    $i+1$ particles  $\{1,..,i+1\}$. Since, by the  construction of the   model,  every  $i\in \N,$ can receive synaptic weights only from  finite neurons  \[K^{i}\subset \{1,...,i-1,i+1\},\] the behavior of any neuron $i$ depends exclusively on a finite system of neurons $j\leq i+1$. In that way,   $\Lambda^{m,i}(u)=\E^{\{1,..,i+1\}}[e^{u(x_{i})}]  $,  can be studied as   the case of the neuron $i$ belonging to a finite system of neurons $\{1,...,i+1\}$. It is important at this point to clarify that  ergodicity in our context does not refer to ergodicity for  the whole infinite-dimensional system but to the individual neuron. For this reason, the stationarity of only  individual neurons $i$ is studied through the RMF limit $\Lambda^i=\lim_{m\rightarrow \infty} \Lambda^{m,i}$. The regeneration argument used in the following to establish ergodicity is after all based on the fact that its neuron is influenced by a finite set of neurons, which themselves are influenced by a smaller set of neurons. In that way, the regeneration time for every neuron is finite. 
 Then, we can obtain  uniform results on $i$ and so conclude for all $i\in \N$.  In this case, $\E^{\{1,..,i+1\}}$ is a measure that is obtained in a probability space of $i+1$ random variables $\{1,...,i+1\}$, with boundary conditions the values of the rest of the r.v. $\{j:j\geq i+2\}$, which however, by the construction of the network  do not affect the first $i+1$ particles. 
  \end{remark}

We can now prove the Harris ergodicity.
\begin{theorem}
    For every $i\in \N,$ and  $m\in \{1,...,M\}$  the system of  $i+1$ neurons $\{x_{m,1}(t),x_{m,2}(t),...,x_{m,i+1}(t)\}_{t\in \R}$  is ergotic.
\end{theorem} 
    \begin{proof}
 At first we notice that according to \cite{M-T}, the   Foster-Lyapunov drift condition of the  last proposition, implies that the process is nonexplosive and that the set $R_c^{KM}$ is positive recurrent.

Following \cite{B-T},  Hypothesis (H1), implies    that for large $c$, the set $R_c^{MK}$, is also a regeneration set, since $\min_{i\in K} \inf_t \phi (x_{i}(t))>\min_{i \in K}r_i>0$.  We can conclude the same  for Hypothesis (H2), since according to Robert and Touboul \cite{r48}, regeneration is guaranteed when neurons spike consecutively and spontaneously, which is the case when     the spike intensity is always strictly positive, as in the current case where $\min_{i\in K}\inf_t\phi(x_{i}(t))>\min_{ i\in K}z>0$.

 The non-explosiveness of the Markov dynamics, together with the fact that the set $R_c^{MK}$ is positive recurrent and a regeneration set, implies the Harris ergodicity of the Markov chain $\{x_{m,1}(t),x_{m,2}(t),...,x_{m,i+1}(t)\}_{t\in \R}$. 
 
 \end{proof}
Having established the  ergodicity of the process, we can now obtain the ODE that describes  the stationary measure. For some $z\geq 0$, define 
\[\Lambda^{m,i}(u):=\E^{\{1,..,i+1\}}[e^{u(x_i^m+z)}],\] where $\E^{\{1,..,i+1\}}$ the expectation with respect to the stationary  measure. We will derive an ODE for the RMF limit  of $\Lambda^{m,i}$, that is
\[ \Lambda^{ i}:=\lim_{m\rightarrow \infty} \Lambda^{m,i}.\]
 It should be noted that in \cite{B-T} they dealt with the more general moment generating function  $\E[e^{u\sum_{i\in K}x_i^m}]$, from which after using the Poisson Hypothesis they obtained an ODE for $\Lambda^{i}$ and consequently an analytic expression for $\beta_i^{1}=\E[N_i(0,1]]$.  In our case, starting directly from $\Lambda^{m,i}$ will allow us to reproduce the result from  \cite{B-T} following the same  exact  proof, but slightly more simplified, as  due to the occurrence of only one random variable involved  we will avoid the complicated calculation of the quantities vanishing at the limit. Furthermore, we have perturbed the density function by a positive  constant  $z>0$, so that  we can consider the Hypothesis (H2) where we obtain the regeneration even when  $r_i=0$.

  \begin{proposition} \label{rmfprop}Assume $g(x_i)=x_i$ and $\phi(x_i)=x_i+z$, for some $z\geq 0$. For every $m\in \{1,...,M\}$ and $i\in   \N$, $\Lambda^{i}$ satisfies the following ode
\begin{align*}
0=&- (1+a_iu) \frac{d}{du}\Lambda^{i}(u)(x) +\sum_{j\neq i}  \left( e^{uw_{ji}} - 1 \right)  \beta_j \Lambda^{i}(u)+\beta_i e^{ur_i}. 
\end{align*}
where $\beta_i=\E^{\{1,..,i+1\}}[x_{i}]+z$.
\end{proposition}
\begin{proof} 
For $m\in \{1,...,M\}$ and $i\in  \N$, for the function $  V^{\{i\},\{m\}}_{u }(x)=\exp^{u(x_i^m+z)}$, we have
\begin{align*} 
\Ge^{M} [V^{\{i\},\{m\}}_{u }](x)=&- a_ig(x_i^m)uV^{\{i\},\{m\}}_{u}(x) +\ \phi (x_i^m)\left[ e^{ur_i-u x_i^m } - 1 \right]V^{\{i\},\{m\}}_{u}(x) 
+\\  &+\sum_{j\neq i} \sum_{n\neq m} \frac{1}{\vert V_{n,j} \vert} \phi( x_j^n)\left[ e^{ uw_{ji}} - 1 \right]V^{\{i\},\{m\}}_{u}(x). 
\end{align*}
Since $\E^{\{1,..,i+1\}}[\Ge^{M} [V^{\{i\},\{m\}}_u]]=0,$ we obtain
\begin{align*}
0=&- (1+a_iu)\frac{d}{du}\Lambda^{m,i}(u)(x) +   \frac{d}{du}\Lambda^{m,i}(u)(x) \I_{u=0}e^{ur_i}
+\\  &+\sum_{j\neq i} \sum_{n\neq m} \frac{1}{\vert V_{n,j} \vert} \left[ e^{ uw_{ji}} - 1 \right] \E^{\{1,..,i+1\}}[(x^{n}_{j}+z)V^{\{i\},\{m\}}_u(x)]. 
\end{align*}
Since  $\vert V_{n,j} \vert$ depends only on $j$, the last term becomes  
\begin{align*}& \sum_{j\neq i} \sum_{n\neq m} \frac{1}{\vert V_{n,j} \vert} \left[ e^{ uw_{ji}} - 1 \right] \E^{\{1,..,i+1\}}[(x_j^n+z)V^{\{i\},\{m\}}_u(x)] = \\  = &\sum_{j\neq i}  \left( e^{uw_{ji}} - 1 \right)\frac{1}{\vert V_{n,j} \vert} \sum_{n\neq m}  \E^{\{1,..,i+1\}}[(x_j^n+z)V^{\{i\},\{m\}}_u(x)].
\end{align*}
We finally get 
\begin{align*}
0=&- (1+a_iu)\frac{d}{du}\Lambda^{m,i}(u)(x) +   \frac{d}{du}\Lambda^{m,i}(u)(x) \I_{u=0}e^{ur_i}
+\\  &+\sum_{j\neq i}  \left( e^{uw_{ij}} - 1 \right)\frac{1}{\vert V_{n,j} \vert} \sum_{n\neq m}  \E^{\{1,..,i+1\}}[(x_j^n+z)V^{\{i\},\{m\}}_u(x)] ,
\end{align*}
or equivalently
\begin{align*}
0=&- (1+a_iu)\frac{d}{du}\Lambda^{m,i}(u)(x) +   \E^{\{1,..,i+1\}}[x_i^m+z]e^{ur_i}+\\  &+\sum_{j\neq i}  \left( e^{uw_{ji}} - 1 \right)\frac{1}{\vert V_{n,j} \vert} \sum_{n\neq m}  \E^{\{1,..,i+1\}}[(x_j^n+z)V^{\{i\},\{m\}}_u(x)]. 
\end{align*}
By Poisson Hypothesis, for   $m\neq n$, 
\begin{align*}
\lim_{M\rightarrow \infty}\E^{\{1,..,i+1\}}[(x_j^n+z)e^{u(x_i^m+z)}]&=\lim_{M\rightarrow \infty}\E^{\{1,..,i+1\}}[x_j^n+z]\E^{\{1,..,i+1\}}[e^{u(x_i^m+z)}]= \\ &=\beta_j \Lambda^{i}.
 \end{align*}
We finally obtain 
\begin{align*}
0=&- (1+a_iu) \frac{d}{du}\Lambda^{i}(u)(x) +\sum_{j\neq i}  \left( e^{uw_{ji}} - 1 \right)  \beta_j \Lambda^{i}+\beta_i e^{ur_i} . 
\end{align*}
\end{proof}

 The main result of this  section about the linear G-L model follows. It should be noted that the main condition (\ref{ConNoNeg}) does not allow purely inhibitory models, that is, with all weights
$w_{ji}\leq 0$. But the inequality still allows some weights to be negative as  long as the terms of the sum with positive weights compensate.\begin{theorem}\label{theorem3} Assume the linear G-L model satisfying  (H1) or (H2). If  the RMF limit   $\beta_i=\E^{\{1,..,i+1\}}[\phi(x_{i})]$  satisfies  
 \begin{align}\label{ConNoNeg}\sum_{j\in K^{i} }  \left( 1-e^{-\frac{w_{ji}}{a_i}}  \right)  \beta_j>0,\end{align}
for any $i\in \N$, then $\beta_i$ solves the following system of equations
\begin{equation}\label{res0T3}\frac{1}{\beta_i}=\int_{-\infty}^0\exp \left( l_i(u)-\sum_{j< i}\beta_j h_{ij}(u)
 \right)du\end{equation}
 where the function $l_i(u)$ and $h_{ij}(u)$ are 
 \begin{equation}\label{res1T3}l_i(u)= \frac{r_i}{a_i}(e^{a_iu}-1)\end{equation}
 \begin{equation}\label{res2T3}h_{ij}(u)= e^{-\frac{w_{ji}}{a_i}} \left( Ei(\frac{w_{ji}}{a_i}e^{ua_i})-Ei(\frac{w_{ji}}{a_i})\right)-u \end{equation}
 and where  $Ei$ denotes the exponential integral function. \end{theorem}
A useful corollary follows.
\begin{corollary}\label{rmfCOR}Assume the linear G-L model of (H1) or (H2)  and that the synaptic weights  $\sup_{i,j\in \N}\vert w_{ji}\vert <\infty$. If \[\sum_{j\in K^{i}}\beta_j<\infty,\] then (\ref{res0T3})   with $l_i,h_{ij}$ as in (\ref{res1T3}) and  (\ref{res2T3}) hold.
\end{corollary}
\begin{remark}\label{rmfREM} If we cancel the drift, i.e. take the limit $a_i\rightarrow 0$, then (see again \cite{B-T}) we get the following explicit expressions of $l_i$ and $h_{ij}$
\[l_i(u)=r_iu  \text{  \     and   \  }h_{ij}(u)=\frac{e^{w_{ji}u}-1}{w_{ji}}-u.\]
\end{remark}
Before we present the proof of Theorem \ref{theorem3}, we present some examples of the linear G-L  networks that go beyond the PUS condition and satisfy   the hypothesis of Theorem \ref{theorem3}.
\begin{example} Some examples that satisfy the conditions of Theorem \ref{theorem3} and at the same time go beyond the PUS property. 

 \underline{Example 1:} Assume (H1). We will consider  the case where   there is not drift, i.e. $a_i=0,\forall i\in \N$. Furthermore, for simplicity, assume that every $i$ receives weights only from $j<i$. Then, as pointed on Remark \ref{rmfREM}, the functions $l_i$ and $h_{ij}$ get an explicit expression which leads to the following simple form of the equation (\ref{res0T3}) 
 \begin{equation}\label{laRMFeq}\frac{1}{\beta_i}=\exp \left(\frac{\sum_{j< i}\beta_j }{w_{ji}}\right)\int_{-\infty}^0\exp\left( (r_i+\sum_{j< i}\beta_j)u\right)\exp \left(-\sum_{j< i}  \frac{\beta_je^{w_{ji}u}}{w_{ji}}
 \right)du.\end{equation}
 We assume that for all $i,j\in \N$, $w_{ij}\geq 0$. 
 
  Since, because of Corollary  \ref{rmfCOR}, it is sufficient to assume $\sum_{i\in \N}\beta_i<\infty$, we need $\beta_i$'s to decrease on $i$ sufficiently fast. But from the construction of the model, we know that every neuron $i$ receives synaptic weights only from the neurons $j<i$. In that way, we guarantee that a neuron  $i$ affects only the neurons $j$ that follow it, $j>i$, and is not affected by the neuron after it $j>i$. In that way,  $\beta_i$ depends only on the behavior of the neurons $j<i$, while  the values of $\{\beta_j, j<i\}$ do not depend, neither  on the value of $r_i$, nor on the values of the synaptic weights $w_{ji}$. 
  
That, implies, that for fixed $\{\beta_j,j<i\}$,  we can choose inductively the values $w_{ji}$ and $r_i$ in such a way that $\beta_i\leq\frac{1}{2^{i}}$.   This can happen in the following way. As in the previous example, from Remark \ref{rmfREM} we obtain the expression (\ref{laRMFeq}). The right hand side of  (\ref{laRMFeq}) increases as $w_{ij}$ goes to zero. In fact, when $w_{ij}\rightarrow 0$ we have 
\[\frac{\sum_{j< i}\beta_j\left( 1- e^{w_{ji}u} \right) }{w_{ji}}\rightarrow -u\sum_{j< i}\beta_j,\] and so
\[\frac{1}{\beta_i}\uparrow  \int_{-\infty}^0\exp\left( r_i u\right)du=\frac{1}{r_i}\]
which implies that the value of $\beta_i$ can be chosen as small as we desire as long as $r_i$ and $w_{ij}$ are sufficiently small. Since the weights $r_i,w_{ij}$ are chosen small, we can consider the case where,  for any $i\in \N$, the synaptic weights  $w_{ij}>a>0$ for all $j>i$, for some small $a$ uniformly on $i$ and $j$. Then, $\sum_{j}w_{ji}=(i-1)a\rightarrow \infty$ as $i\rightarrow \infty$, and the PUS condition does not hold.

 \underline{Example 2:} Assume (H1). If we consider the linear G-L model  with a drift $a_i>0$, then the same reasoning as in the previous example can be applied to equation (\ref{res0T3}). In that case, the values of $\beta_j$ for $j<i$ do not depend  on $w_{ji}$ and $r_i$ as in "example 2", nor  on $a_i$. Then, the smaller the values of $w_{ij}$, $r_i$ are and the bigger the value of $a_i$ is, the smaller the value of $\beta_i$.
Then, the PUS condition will not hold, as in the previous example, as long as we choose for any $i\in \N$, the   weights  $w_{ij}>a>0$ for all $j>i$.\end{example}
 \underline{The proof of Theorem \ref{theorem3}: }
\begin{proof}
The proof of the Theorem \ref{theorem3} follows directly from Theorem 3.6 from \cite{B-T}. The authors base the proof of their result on the following proposition: 
 \begin{proposition}  (Proposition 5.2 in \cite{B-T} ) If $\tau>0$ and $f(-\tau)>0$, then the ODE
\begin{equation}\label{B-Tode}(1+\frac{u}{\tau})\ \Lambda'(u)+f(u)\Lambda(u)-g(u)=0\end{equation}
admits a unique continuous solution 
\[\Lambda(u)=\int_{-\tau}^u e^{-\int_v^u\frac{f(\omega)}{1+\omega/\tau}d\omega}\frac{g(v)}{1+v/\tau}dv.\]
\end{proposition}
One observes that for  functions $g(u)=-\beta_i e^{ur_i}$ and $f(u)=-\sum_{j\neq i}  \left( e^{uw_{ji}} - 1 \right)  \beta_j$, and  $\tau=\frac{1}{a_i} $,   the ODE (\ref{B-Tode})  has  already  been  proven  in Proposition \ref{rmfprop}. It remains to impose conditions on the linear G-L model so that the  condition  $f(-\frac{1}{a_i})>0$ is satisfied, that is 
\[\sum_{j\neq i}  \left( e^{-\frac{1}{a_i}w_{ji}} - 1 \right)  \beta_j<0.\]
 \end{proof}

\section{Non-linear intensity and non linear interaction. }\label{RMF2}

In the current section, we will consider cases that go beyond  the case of linear $\phi$ and $g$.  To do so, we will introduce a new type of slightly modified interactions
as defined below in (\ref{gen3.2}). The generator of the process,  for any test function $ f : \R_+^N \to \R $  and $x \in \R_+^N$  is   as follows
    \begin{align} \label{gen3.1}
\Ge f (x ) =&-\sum_{i}a_ig(x_i)\frac{d}{dx_i}f +\sum_{i }   \phi (x_i) \left[ f ( \Delta_i ( x)  ) - f(x) \right],  
\end{align}
for $g(x_i)\geq 0$ for all $x_i\geq 0$, and $a_i\geq 0$ for all $i$
where
\begin{equation}\label{gen3.2}
(\Delta_i (x))_j =    \left\{
\begin{array}{ll}
\left((x_j )^{r}+(w_{i j})^r \right)^\frac{1}{r} & j \neq i \\
r_{i}  & j = i 
\end{array}
\right\}.\end{equation}
 
As in the case of the linear G-L model, we consider two   different sets of hypothesis, one with an    intensity function that is always strictly positive and one where it is not, in which case a reset value of non zero is required for the neuron that spikes.  \begin{itemize}

\item   \underline{Hypothesis (H3):} For some $r>0$  and $r_i>0$,
 \[ \phi(x_i)=(x_i)^r,  \ g(x_i)=x_i     \  \text{and} \ \min_i r_i>0. \]

\item   \underline{Hypothesis (H4):} For some $r>0$ and $z\geq 0$
 \[ \phi(x_i)=(x_i)^{r}+z,  \ g(x_i)=x_i  . \]
 \end{itemize}
 
 The main result of this section about the non linear  model with non linear interactions follows:
\begin{theorem}\label{theorem4} Assume the non linear model (\ref{gen3.1})-(\ref{gen3.2})  with intensity function and drift as in (H3) or (H4). If  the RMF limit   $\beta_i=\E^{\{1,..,i+1\}}[\phi(x_i)]=\E^{\{1,..,i+1\}}[(x_i)^{r}]+z$, for $r>0$ and $z\geq 0$,  satisfies  
 \[\sum_{j\in K^{i} }  \left( 1-e^{-\frac{W^{r}_{ji}}{a_i}}  \right)  \beta_j>0,\]
for any $i\in \N$, then $\beta_i$ solves the following system of equations
\[\frac{1}{\beta_i}=\int_{-\infty}^0\exp \left( l_i(u)-\sum_{j< i}\beta_j h_{ij}(u)
 \right)\]
 where the functions $l_i(u)$ and $h_{ij}(u)$ are 
\[l_i(u)= \frac{r^{r}_i}{a_i}(e^{a_iu}-1)\]
and
 \[h_{ij}(u)= e^{-\frac{w_{ji}}{a_i}} \left( Ei(\frac{W^{r}_{ji}}{a_i}e^{ua_i})-Ei(\frac{W^{r}_{ji}}{a_i})\right)-u \]
   where  $Ei$ denotes the exponential integral function. \end{theorem}
 The proof of the theorem is identical to the proof of Theorem \ref{theorem3}. We only need to check the analogue of Proposition \ref{rmfprop} for the non linear interactions and intensity function. 
In analogy to the previous subsection, we now define, for some $z\geq 0$ and $r>0$,  
 \[\Lambda^{m,i}(u)=\E^{\{1,..,i+1\}}[e^{u((x_i^m)^{r}+z)}],\]
  where $\E^{\{1,..,i\}}$ the expectation with respect to the stationary measure. Similarly, the RMF limit  of $\Lambda^{m,i}$ is
\[ \Lambda^{ i}:=\lim_{m\rightarrow \infty} \Lambda^{m,i}.\]
\begin{proposition} Assume $\phi(x)=x^{r}+z$ and $g(x)=x $. For  every $m\in \{1,...,M\}$ and $i\in \{1,...,K\}$, $\Lambda^{i}$ satisfies the following ode
\begin{align*}
0=&- (1+a_iu) \frac{d}{du}\Lambda^{i}(u)(x) +\sum_{j\neq i}  \left( e^{uw_{ji}} - 1 \right)  \beta_j \Lambda^{i}+\beta_i . 
\end{align*}
where $\beta_i=\E^{\{1,..,i+1\}}[(x_i)^{r}]+z$,
\end{proposition}
\begin{proof} 
For $m\in \{1,...,M\}$ and $i\in \{1,...,K\}$, consider the function $  V^{\{i\},\{m\}}_{u,r}(x)=\exp^{u \left((x_i^m)^{r}+z\right)}$. We have
\begin{align*} 
\Ge^{M} [V^{\{i\},\{m\}}_{u,r}](x)=&- a_ig(x_i^m)uV^{\{i\},\{m\}}_{u,r}(x) +\ \phi (x_i^m)\left[ e^{-u( x_i^m)^r } - 1 \right]V^{\{i\},\{m\}}_{u,r}(x) 
+\\  &+\sum_{j\neq i} \sum_{n\neq m} \frac{1}{\vert V_{n,j} \vert} \phi( x_j^n)\left[ V^{\{i\},\{m\}}_{u,r}(\Delta_i(x)) - V^{\{i\},\{m\}}_{u,r}(x) \right]. 
\end{align*}
But under the new interaction regime introduced in (\ref{gen3.2}) we now have 
\[ V^{\{i\},\{m\}}_{u,r}(\Delta_i(x)) = e^{u\left( (x_i^m)^{r}+ z+(w_{ji})^r\right)} \]
and so
\begin{align*} 
\Ge^{M} [V^{\{i\},\{m\}}_{u,r}](x)=&- a_ig(x_i^m)uV^{\{i\},\{m\}}_{u,r}(x) +\ \phi (x_i^m)\left[ e^{-u( x_i^m)^r } - 1 \right]V^{\{i\},\{m\}}_{u,r}(x) 
+\\  &+\sum_{j\neq i} \sum_{n\neq m} \frac{1}{\vert V_{n,j} \vert} \phi( x_j^n)\left[  e^{u  w_{ji}^r} -1 \right] V^{\{i\},\{m\}}_{u,r}(x). 
\end{align*}
Since $\E^{\{1,..,i+1\}}[\Ge^{M} [V^{\{i\},\{m\}}_{u,r}]]=0$ we obtain
\begin{align*}
0=&- (1+a_iu)\frac{d}{du}\E^{\{1,..,i+1\}}[ V^{\{i\},\{m\}}_{u,r}(x)] +    \frac{d}{du}\E^{\{1,..,i+1\}}[V^{\{i\},\{m\}}_{u,r}(x)] \I_{u=0}e^{ur^{r}_i}
+\\  &+\sum_{j\neq i} \sum_{n\neq m} \frac{1}{\vert V_{n,j} \vert} \left[ e^{ uw^r_{ji}} - 1 \right] \E^{\{1,..,i+1\}}[\left((x_j^n)^r+z\right)V^{\{i\},\{m\}}_{u,r}(x)]. 
\end{align*}
Since  $\vert V_{n,j} \vert$ depends only on $j$, the last term becomes  
\begin{align*}& \sum_{j\neq i} \sum_{n\neq m} \frac{1}{\vert V_{n,j} \vert} \left[ e^{ uw^r_{ji}} - 1 \right] \E^{\{1,..,i+1\}}[(x_j^n)^rV^{\{i\},\{m\}}_{u,r}(x)] = \\  = &\sum_{j\neq i}  \left( e^{uw^r_{ji}} - 1 \right)\frac{1}{\vert V_{n,j} \vert} \sum_{n\neq m}  \E^{\{1,..,i+1\}}[\left((x_j^n)^r+z\right)V^{\{i\},\{m\}}_{u,r}(x)]
.\end{align*}
We finally get 
\begin{align*}
0=&- (1+a_iu)\frac{d}{du}\E^{\{1,..,i+1\}}[ V^{\{i\},\{m\}}_{u,r}(x)]+   \frac{d}{dx_i^m}\E^{\{1,..,i+1\}}[ V^{\{i\},\{m\}}_{u,r}(x)] \I_{u=0}e^{ur^{r}_i}
+\\  &+\sum_{j\neq i}  \left( e^{uw_{ij}} - 1 \right)\frac{1}{\vert V_{n,j} \vert} \sum_{n\neq m}  \E^{\{1,..,i+1\}}[\left((x_j^n)^r+z\right)V^{\{i\},\{m\}}_u(x)] ,\end{align*}
or equivalently
\begin{align*}
0=&- (1+a_iu)\frac{d}{du}\Lambda^{m,i}(u)(x) +   \E^{\{1,..,i+1\}}[x_i^m]e^{ur^{r}_i}+\\  &+\sum_{j\neq i}  \left( e^{uw_{ji}} - 1 \right)\frac{1}{\vert V_{n,j} \vert} \sum_{n\neq m}  \E^{\{1,..,i+1\}}[\left((x_j^n)^r+z\right)V^{\{i\},\{m\}}_u(x)]. 
\end{align*}
By Poisson Hypothesis, when $M$ goes to infinity and $m\neq n$, 
\begin{align*}
\lim_{M\rightarrow \infty} \E^{\{1,..,i+1\}}[(x_j^n)^re^{u(x_i^m)^r}]=&\lim_{M\rightarrow \infty} \E^{\{1,..,i+1\}}[(x_j^n)^r]\E^{\{1,...,i+1\}}[e^{u(x_i^m)^r}]=\\ =&\E^{\{1,..,i+1\}}[(x_j)^r]\E^{\{1,...,i+1\}}[e^{u(x_i)^r}].
 \end{align*}
This leads to
\begin{align*}
0=&- (1+a_iu) \frac{d}{du}\E^{\{1,..,i+1\}}[e^{u(x_i)^r}] +\\  +\sum_{j\neq i} & \left( e^{uw^r_{ji}} - 1 \right)  \E^{\{1,..,i+1\}}[(x_j)^r+z]\E^{\{1,..,i+1\}}[e^{u(x_i)^r}]+\E^{\{1,..,i+1\}}[(x_i)^r+z] e^{ur^{r}_i}. 
\end{align*}
\end{proof}
 As in the previous section where the linear G-L model was studied, similar examples can be constructed.

\end{document}